\documentclass[11 pt,fullpage]{article}
\usepackage{amsfonts,amssymb}
\usepackage[hidelinks]{hyperref}
\usepackage{lipsum,amsmath,amsthm}
\usepackage[numbers,sort&compress]{natbib}
\usepackage{ dsfont }
\usepackage{appendix}
\usepackage{color}
\usepackage{esint}
\usepackage{hyperref}
\usepackage{mathtools}
\usepackage{scalerel}
\usepackage{comment}
\usepackage{subfigure}
\mathtoolsset{showonlyrefs}
\usepackage[font=small]{caption}
\usepackage[right = 2.5cm, left=2.5cm, top = 2.5cm, bottom =2.5cm]{geometry}
\usepackage[utf8]{inputenc}
\usepackage[english]{babel}

\theoremstyle{plain}
\newtheorem{theorem}{Theorem}

\newtheorem{proposition}{Proposition}[section]
\newtheorem{lemma}[proposition]{Lemma}

\theoremstyle{definition}
\newtheorem{remark}{Remark}[section]

\usepackage{todonotes}
\allowdisplaybreaks

\numberwithin{equation}{section}

\newcommand{\grad}{\nabla}

\newcommand{\R}{\mathbb{R}}
\newcommand{\N}{\mathbb{N}}

\newcommand{\C}{\mathbb{C}}

\newcommand{\Z}{\mathbb{Z}}
\newcommand{\T}{\mathbb{T}}

\newcommand{\dee}{\mathrm{d}}

\newcommand{\dt}{\dee t}

\newcommand{\dx}{\mathrm{d} x}

\newcommand{\dy}{\dee y}

\usepackage{bbm}

\newtheorem*{lemma*}{Lemma}

\theoremstyle{definition}

    \def\XXint#1#2#3{{\setbox0=\hbox{$#1{#2#3}{\int}$}
         \vcenter{\hbox{$#2#3$}}\kern-.5\wd0}}

\DeclareUnicodeCharacter{2BC}{'}

\usepackage{chngcntr}
\counterwithout{equation}{subsection}
\counterwithout{equation}{section} 

\begin{document}

 \title{Norm Growth, Non-uniqueness, and Anomalous Dissipation in Passive Scalars}
	\author{Tarek M. Elgindi and Kyle Liss}
	
	\maketitle

\begin{abstract}
We construct a divergence-free velocity field $u:[0,T] \times \T^2 \to \R^2$ satisfying $$u \in C^\infty([0,T];C^\alpha(\T^2)) \quad \forall \alpha \in [0,1)$$ such that the corresponding drift-diffusion equation exhibits anomalous dissipation for every smooth initial data. We also show that, given any $\alpha_0 < 1$, the flow can be modified such that it is uniformly bounded only in $C^{\alpha_0}(\T^2)$ and the regularity of solutions satisfy sharp (time-integrated) bounds predicted by the Obukhov-Corrsin theory. The proof is based on a general principle implying $H^1$ growth for all solutions to the transport equation, which may be of independent interest.
\end{abstract}

\section{Introduction}

We consider the advection-diffusion equation on $\T^2 = \R^2 \slash (2\pi\Z^2)$:

\begin{equation} \label{eq:ADE}
\begin{cases}
\partial_t f^\kappa + u \cdot \grad f^\kappa = \kappa \Delta f^\kappa, \\ 
f^\kappa|_{t=0} = f_0.
\end{cases}
\end{equation}
Here, for some time $T > 0$, $u:[0,T]\times \T^2 \to \R^2$ is a given divergence-free velocity field, $f^\kappa:[0,T]\times \T^2 \to \R$ is a passive scalar representing, for instance, temperature or concentration, $\kappa > 0$ is a small constant, and $f_0 \in L^2$ is a mean-free initial data. The vector field $u$ may be prescribed as a solution to some hydrodynamical equation, like the Euler or Navier-Stokes equations, or it may simply be imposed. Since $u$ is divergence free, the $L^2$ energy of a solution to \eqref{eq:ADE} is monotone decreasing and governed by the energy balance
\begin{equation} \label{eq:energybalance}
\frac{d}{dt} \|f^\kappa(t)\|_{L^2}^2 = -2\kappa \|\grad f^\kappa(t)\|_{L^2}^2,
\end{equation}
or equivalently 
\begin{equation}
\|f^\kappa(t)\|_{L^2}^2=\|f^\kappa(0)\|_{L^2}^2-2\kappa\int_{0}^t\|\grad f^\kappa(s)\|_{L^2}^2\dee s,
\end{equation}
on any interval $[0,t]$ where $f^\kappa$ is sufficiently smooth. In particular, the quantity \[\mathcal{E}_\kappa(t):=2\kappa\int_{0}^t\|\grad f^\kappa(s)\|_{L^2}^2 \dee s\] determines the energy dissipation of a solution. The size of this quantity is, in turn, related to the distribution of the solution in Fourier space or its average length scale in physical space. 

Even though the velocity field does not enter directly into \eqref{eq:energybalance}, it generally plays an important role in the rate of energy dissipation. Indeed, advection typically contributes to the formation of small spatial scales and consequently enhances the $L^2$ decay of the scalar. In fact, if $u$ is rougher than Lipschitz (as is expected in regimes of turbulent advection), then it is possible for this effect to be so dramatic that a fixed amount of dissipation can occur with arbitrarily small diffusion and in a $\kappa$-independent length of time. That is, one can have \textit{anomalous dissipation}:
\begin{equation}\label{Anomaly}\liminf_{\kappa\rightarrow 0} \mathcal{E}_{\kappa}(t)=c_0>0
\end{equation}
for some $t > 0$. The term anomalous dissipation refers to how \eqref{Anomaly} implies that solutions to \eqref{eq:ADE} dissipate energy at a rate that is independent of the diffusivity constant $\kappa>0.$ The presence of anomalous dissipation in the advection-diffusion and Navier-Stokes equations is a central assumption in the phenomenological theories of turbulence, playing a fundamental role in the Obukhov-Corrsin theory for passive scalars  \cite{Obukhov49, Corrsin51,Sr19, SS2000} as well as the K41 hydrodynamical theory \cite{K41a, K41b, Frisch}. The predicted dissipation anomalies for both turbulent fluids and scalars advected by them are well supported by numerics and experiments \cite{Sr19, DSrY05, Kaneda03, Pearson02}.

It is easy to achieve anomalous dissipation mathematically if the initial $f_0=f_0^\kappa$ become rough when $\kappa\rightarrow 0$ (even with $u\equiv 0$) and it is also easy to achieve when the velocity field becomes unbounded pointwise, even if $f_0$ is fixed. Certainly, if $f_0$ is chosen independent of $\kappa>0$ and if $u$ is smooth, \eqref{Anomaly} is impossible for finite $t>0.$ This can be seen easily since smoothness of $u$ propagates smoothness of $f_0$ ($L^2$ compactness of the solution being the property of interest). Anomalous dissipation is similarly impossible when $u$ belongs to certain DiPerna-Lions classes \cite{DL} or, more generally, when the transport equation (\eqref{eq:ADE} with $\kappa=0$) has unique $L^2$ solutions. 

An example of a deterministic velocity field that exhibits (genuine) anomalous dissipation for a large class of initial data was constructed in \cite{DEIJ22} using alternating ``sawtooth'' shear flows. The main purpose of this paper is to revisit the idea of \cite{DEIJ22} and show that alternating shears can in fact be used to achieve anomalous dissipation \emph{for all} sufficiently smooth initial data, while at the same time improving upon the regularity of the earlier construction.

\subsection{Main results}

Our first result is an explicit example of a divergence-free velocity field $u:[0,T] \times \T^2 \to \R^2$ that is uniformly smooth in time in $C^\alpha(\T^2)$ for every $\alpha < 1$ and for which the solution of \eqref{eq:ADE} exhibits anomalous dissipation for every mean-zero and smooth initial data.

\begin{theorem} \label{thrm:main}
Fix $T > 0$. There exists a divergence-free velocity field $u:[0,T] \times \T^2 \to \R^2$ with
\begin{equation}\label{eq:generalizedregularity}
    u \in C^\infty([0,T]; C^\alpha(\T^2)) \quad \forall \alpha \in (0,1)
\end{equation}
such that, for every mean-zero and smooth initial data $f_0,$ the solutions $f^\kappa$ exhibit anomalous dissipation on $[0,T]$. Solutions to the corresponding transport equation are non-unique while $u$ satisfies \begin{equation}\label{eq:velocityregularity}|u(t,z)-u(t,z')|\leq C\omega(|z-z'|)\end{equation} for all $t\in [0,T]$ and $z,z'\in\mathbb{T}^2$, with $\omega(s)=s(1+|\log(s)|^4)$ and $C>0$ a universal constant.  
\end{theorem}

\noindent We now make a few remarks on the above results. 

\begin{remark} (Modulus of continuity of the velocity field)
It is straightforward to see from the proof that the power 4 in the definition of $\omega$ can be replaced with any $p > 3$, but we do not do this for simplicity of presentation. Modifying slightly the proof, we can bring the power on the logarithm down to any $p>2$, at least for suitable initial data. It seems that our proof cannot go all the way down to the Osgood threshold. As anomalous dissipation is known to imply non-uniqueness for the underlying transport equation with $\kappa = 0$ (see e.g. \cite{DEIJ22}), an interesting open question is if for any non-Osgood modulus of continuity one can construct a velocity field enjoying that modulus of continuity uniformly in time and such that the corresponding drift-diffusion equation exhibits anomalous dissipation or such that the corresponding transport equation exhibits non-uniqueness.
\end{remark}

\begin{remark} (Regularity of the initial data) \label{rem:constants}
    It is not required in Theorem~\ref{thrm:main} that $f_0$ be smooth. All that we require is $f_0 \in H^{1+s} \cap W^{1,\infty}$ for some $s > 2/5$. Moreover, for a fixed such $s$, the amount of energy dissipated depends only on upper bounds for 
$$\|f_0\|_{W^{1,\infty}}/\|f_0\|_{L^2} \quad \text{and} \quad \frac{\|f_0\|_{L^2}^s \|f_0\|_{\dot{H}^{1+s}}}{\|f_0\|_{H^1}^{1+s}}. $$
For any $\delta > 0$ and with same proof, the requirement $s > 2/5$ can be weakened to $s>\delta$ by allowing the power 4 in the definition of $\omega$ to be sufficiently large depending on $\delta$. 
\end{remark}

\begin{remark} (Dimensions $d\geq 2$)
    Theorem~\ref{thrm:main} provides, as an immediate corollary, anomalous dissipation in any dimension $d \ge 2$. Indeed, one can lift the velocity field to $\T^d$, switching its orientation in space $d-1$ times as time evolves, to construct a divergence-free velocity field $u \in C^\infty([0,T]; C^\alpha(\T^d))$ that exhibits anomalous dissipation for every mean-zero $f_0 \in C^\infty(\T^d)$.
\end{remark}

\begin{remark} (Euler and Navier-Stokes)
Let us remark finally that the velocity field of Theorem \ref{thrm:main} solves the 2d Euler equation with a force that is uniformly smooth in time with values in $C^\alpha.$ This is simply because the velocity field is just a shear flow for each $t\in [0,T]$, so the force is just $\partial_t u$. Note that, upon inspecting the various parameters in the proof, it is easy to use this to construct a so-called $2\frac{1}{2}$-dimensional solution to the 3d Navier-Stokes system that gives anomalous dissipation. 
\end{remark}

By analogy with the scaling of velocity increments over the inertial range predicted by the K41 theory and its connection with the Onsager H\"{o}lder-1/3 regularity threshold for the conservation of kinetic energy in the Euler equations \cite{Eyink94, CET94, Ons49}, the scaling of structure functions over the inertial-convective range within the Obukhov-Corrsin theory of scalar turbulence \cite{Obukhov49,Corrsin51}
underlies a regularity threshold for anomalous dissipation in the advection-diffusion equation. Specifically, if the advecting velocity field satisfies $u \in L^\infty([0,T];C^\alpha)$ for some $\alpha \in [0,1)$ and the family of solutions $\{f^\kappa\}_{\kappa > 0}$ to \eqref{eq:ADE} remain uniformly bounded in $L^2([0,T];C^\beta)$, then heuristic scaling arguments suggest that 
\begin{equation} \label{eq:regthreshold}
    \limsup_{\kappa \to 0} \kappa \int_0^T \|\grad f^\kappa(s)\|_{L^2}^2\dee s = 0 \quad \text{unless} \quad \beta \le \frac{1-\alpha}{2}.
\end{equation}
This statement generalizes to different time integrability exponents (see e.g. the introduction of \cite{CCS}) and can be proven in a similar fashion to the rigidity side of Onsager's conjecture (see \cite{DEIJ22}).

Since the velocity field of Theorem~\ref{thrm:main} belongs to $L^\infty([0,T];C^\alpha(\T^2))$ for every $\alpha \in (0,1)$, it follows from \eqref{eq:regthreshold} that the associated solutions $f^\kappa$ cannot retain any degree of H\"{o}lder regularity (even in a time integrated sense) uniformly as $\kappa \to 0$. In our second result, we show that we can modify the velocity field from Theorem~\ref{thrm:main} so that it remains uniformly bounded only in some fixed H\"{o}lder class and the scalar regularity gets arbitrarily close to the threshold set by \eqref{eq:regthreshold}.

\begin{theorem} \label{thrm:2}
    Fix $T > 0$, $\alpha \in (0,1)$, and $\beta < (1-\alpha)/2$. There exists a divergence-free velocity field $u:[0,T] \times \T^2 \to \R^2$ with $u \in C([0,T];C^\alpha(\T^2))$ such that for every mean-zero $f_0 \in C^\infty(\T^2)$ the solutions $f^\kappa$ of \eqref{eq:ADE} exhibit anomalous dissipation on $[0,T]$ and satisfy 
    \begin{equation} \label{eq:scalarregularity}
    \sup_{\kappa \in [0,1] }\|f^\kappa\|_{L^2([0,T];C^\beta(\T^2))} < \infty.
    \end{equation}
\end{theorem}

\subsection{Previous works and discussion}

We now provide an overview of known results concerning anomalous dissipation and discuss the present contribution within the context of the previous literature.

\subsubsection{Previous work} \label{sec:previousresults}

There have been a number of recent works that consider anomalous dissipation for the advection-diffusion equation with divergence free drift. The first example of a deterministic vector field that exhibits anomalous dissipation was given in \cite{DEIJ22}. Specifically, for any $\alpha \in (0,1)$ and $d\ge 2$, the authors construct a velocity field $u \in L^1([0,1];C^\alpha(\T^d)) \cap L^\infty([0,1]\times \T^d)$ that yields anomalous dissipation for a large class of data. The example is based on alternating approximately piece-wise linear shear flows with rapidly increasing frequencies up to a singular time. One can view the result of \cite{DEIJ22} as being a continuation of previous works on enhanced dissipation \cite{CKRZ,CZDE}. Building off the idea of \cite{DEIJ22}, Bru\'e and De Lellis \cite{BDL} exhibited an example of anomalous dissipation in the forced 3d Navier-Stokes equation. Thereafter, Colombo, Crippa, and Sorella \cite{CCS} revisited the problem of anomalous dissipation and made a number of new contributions. Namely, they showed that vanishing viscosity is \emph{not} a selection criterion for uniqueness in the transport equation. Moreover, they showed that the velocity field can be significantly more regular than the advertised regularity of  \cite{DEIJ22} while maintaining the corresponding sharp upper bounds on the scalar regularity. The example of \cite{CCS} is based on using examples of finite-time mixers constructed, for example, in \cite{ACM}. See also \cite{BCCDS,JS} for follow-up works. Very recently, Armstrong and Vicol \cite{AV} introduced a different mechanism for anomalous dissipation that is not based on a ``finite-time singularity" in the velocity field but on a continuous cascade to high frequencies in the advection-diffusion equation; this is made rigorous using ideas from quantitative homogenization. A key advance in the construction of \cite{AV} is that, for their choice of $u$, \emph{all} solutions to \eqref{eq:ADE} with sufficiently smooth initial data exhibit anomalous dissipation. Additionally, the dissipative anomaly is spread out over time; in particular, the energy of the solution is continuous uniformly in $\kappa$. Finally, in an interesting paper of Huysmans and Titi \cite{HT}, another example of anomalous dissipation is given based on mixing in finite time. A surprising consequence of the analysis in \cite{HT} is the existence of a solution to the transport equation with energy that jumps down and then up again, while also being a limit of vanishing viscosity. 

Anomalous dissipation has also been recently established by Bedrossian, Blumenthal and Punshon-Smith for passive scalars driven by a spatially smooth, white-in-time stochastic forcing and advected by velocity fields solving various stochastic fluid models \cite{BBPSBatchelor} (see also \cite{BBPS21, BBPS22} for the earlier works of the same authors on mixing and enhanced dissipation used importantly in \cite{BBPSBatchelor}). The velocity realizations are almost surely uniformly bounded in $C^1$ on every finite time interval, so anomalous dissipation in the sense of \eqref{Anomaly} is impossible. In this setting, anomalous dissipation refers to a constant, non-vanishing flux of scalar $L^2$ energy from low to high frequencies in statistical equilibrium and the convergence of solutions as $\kappa \to 0$ to a statistically stationary solution of the \emph{forced} transport equation that lives in a regularity class just below some Onsager critical space.

\subsubsection{Discussion}

Let us now take a moment to reflect upon the place of this work in the context of previous works. First, the construction we use here is based on alternating shear flows, inspired by \cite{DEIJ22}. While the velocity field is \emph{not} mixing, solutions to the transport equation do lose compactness in $L^2$ (see Figure~\ref{fig1}). Second, as compared to \cite{CCS}, we are able to recover the previous results on the regularity of the velocity field and the passive scalar and this is done \emph{for all} sufficiently smooth data. The velocity field constructed here is also more regular in space and time than the one constructed in \cite{AV}, where the velocity is $C_{t,x}^{1/3-}.$ We do not consider the question of selection or the lack of selection in the vanishing diffusivity limit. Also, since in the example we give the energy only dissipates at one point in time as $\kappa \to 0$, a drawback of our result is that the energy is discontinuous in the limit (as compared to \cite{AV}, where the energy dissipates continuously). It is possible that suitable modifications of our arguments could give continuous energy dissipation, though it seems to be more difficult to give a construction that yields continuous and strictly decreasing energy in the limit (we are not aware of a construction giving this latter property). 

\begin{figure}[h] 
\centering
\subfigure
{\includegraphics[width=5.25cm]{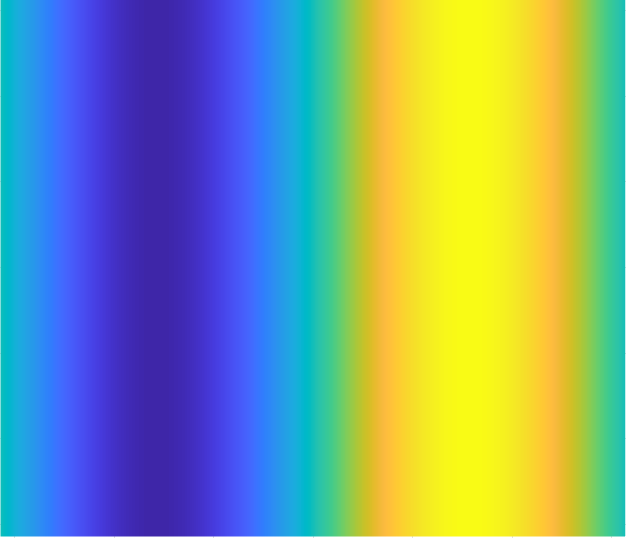}}
\subfigure{\includegraphics[width=5.25cm]{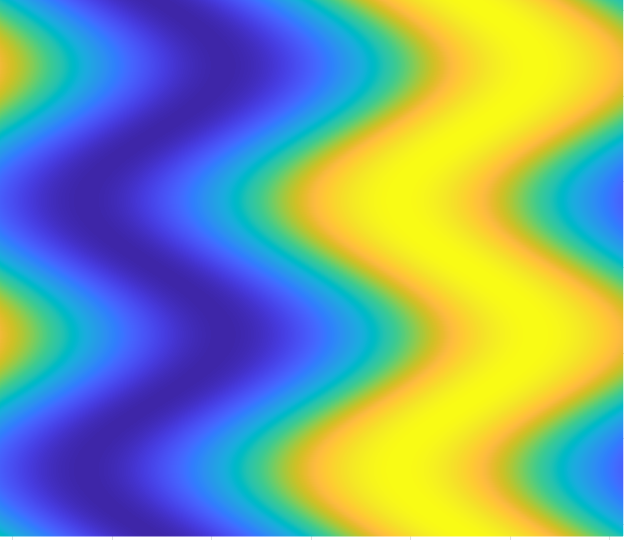}}
\subfigure{\includegraphics[width=5.25cm]{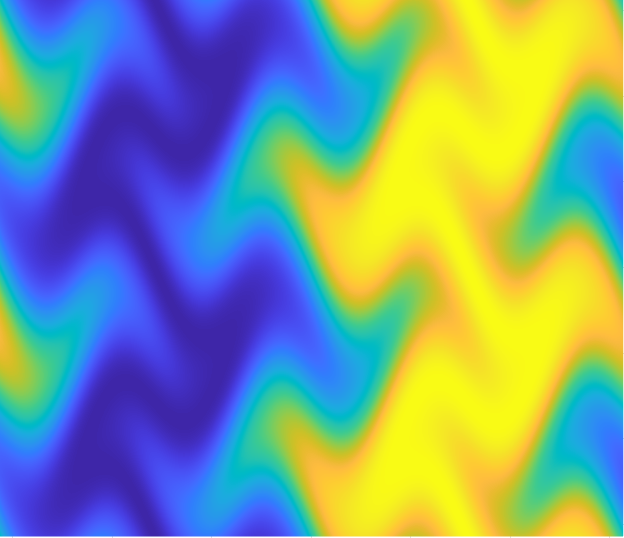}} \\
\subfigure{\includegraphics[width=5.25cm]{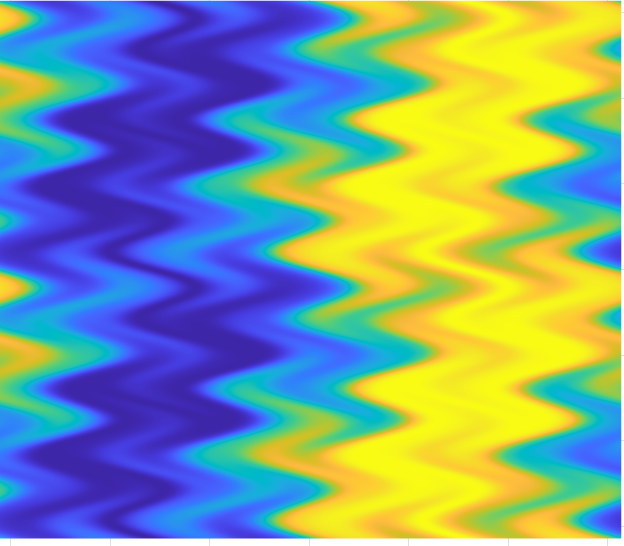}}
\subfigure{\includegraphics[width=5.25cm]{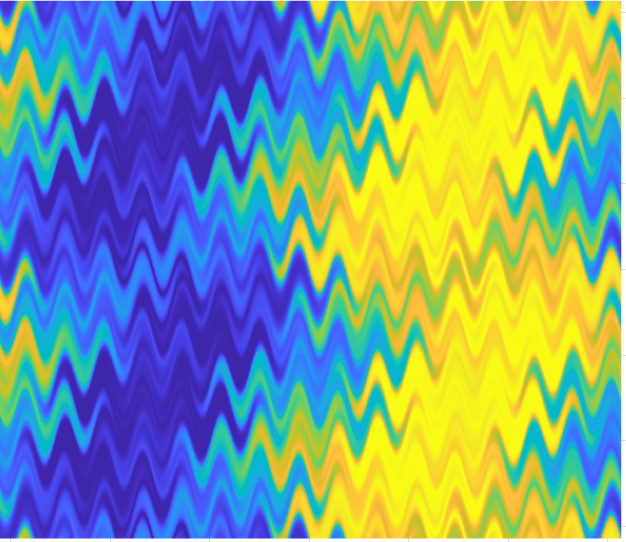}}
\subfigure{\includegraphics[width=5.25cm]{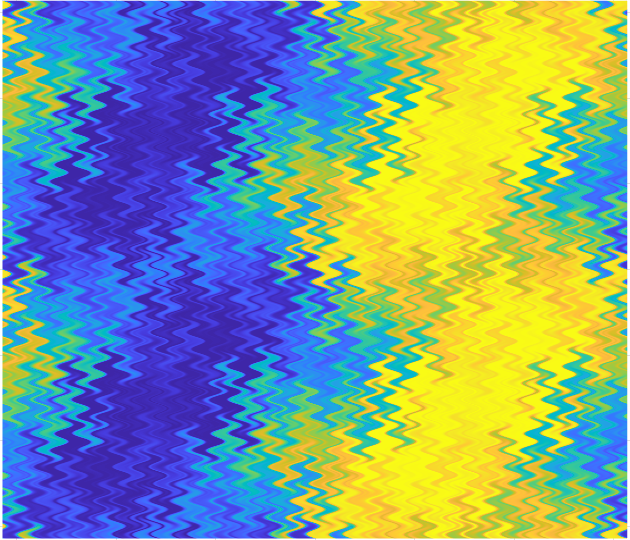}}
\caption{Numerical simulation of the transport equation~\eqref{eq:AE} with $u$ a regularized version of the alternating shear flow used in the proof of Theorem~\ref{thrm:main} and initial condition $f_0(x,y) = \sin(x)$, displayed in the upper left image. Moving from left to right, the scalar is shown after each successive shearing direction is applied.}
\label{fig1}
\end{figure}

\subsection{Main ideas of the proof}

The main idea of \cite{DEIJ22} was that anomalous dissipation (or sharp enhanced dissipation) in the advection-diffusion equation can be deduced under the condition that solutions to just the advection equation 
\begin{equation} \label{eq:AE}
    \begin{cases}
        \partial_t f + u \cdot \grad f = 0, \\ 
        f|_{t=0} = f_0
    \end{cases}
\end{equation}
satisfy a bound of the form
\begin{equation}\label{EnergySpectrum}\frac{\|\mathbb{P}_{>N}f\|_{L^2}}{\|f\|_{L^2}}\geq c\end{equation} with $N=c \|f\|_{H^1}$ for some fixed constant $c>0$ (where $\mathbb{P}_{>N}$ is a Fourier projection). Estimate \eqref{EnergySpectrum} implies that if the $H^1$ norm of $f$ becomes large, then the energy spectrum of $f$ contains some bump localized around frequencies comparable to $\|f\|_{H^1}.$ It is not difficult to show (see \cite{DEIJ22}) that the condition \begin{equation}\label{H1Growth}\lim_{t\rightarrow T_*}\int_0^{t} \|f\|_{H^1}^2=+\infty\end{equation} is \emph{equivalent} to anomalous dissipation when \eqref{EnergySpectrum} is satisfied for a uniform $c>0$ and all $t\in [0,T_*)$. The key is thus to construct a velocity field with the property that all solutions to \eqref{eq:AE} satisfy \eqref{EnergySpectrum}-\eqref{H1Growth}. It is easy to imagine that advecting a scalar by a rougher and rougher velocity field leads to unbounded $H^1$ growth, even fast enough to satisfy \eqref{H1Growth}, but it is not so clear how to construct that velocity field in such a way that this holds while also maintaining \eqref{EnergySpectrum} for \emph{every} smooth solution to \eqref{eq:AE}.  We achieve this by using a variant of the  following general lemma. 
\begin{lemma}\label{fwdbckwdp}(Forwards-Backwards Principle)
Assume that $\Phi$ is a volume preserving and bi-Lipschitz map of a smooth manifold $\Omega.$ Assume that for every mean-zero $f\in H^1(\Omega)$, we have that 
\begin{equation}\label{fwdbckwd}\frac{\|f\circ\Phi\|_{\dot{H}^1}^2+\|f\circ\Phi^{-1}\|_{\dot{H}^1}^2}{2}\geq K\|f\|_{\dot{H}^1}^2, \end{equation} for some $K>1.$ Then, there exists $c=c(f)$ so that 
\[\|f\circ\Phi^{n}\|_{H^1}\geq K^{|n|} c \|f\|_{H^1},\] for all $n\in\mathbb{Z}.$ 
\end{lemma}
The proof simply follows from the fact that $\|f \circ \Phi^n\|_{\dot{H}^1} \ge \|f\|_{L^2}$ for every $n \in \Z$ by the Poincar\'{e} inequality and that if we have $\|f\circ\Phi^{-(n_0-1)}\|_{H^1}\geq \|f\circ \Phi^{-n_0}\|_{H^1}$ for some $n_0 \in \Z$, then the assumption implies that $\|f\circ\Phi^{-n}\|_{H^1}$ is increasing for $n\ge n_0$. 
Indeed, one can interpret \eqref{fwdbckwd} as a type of (discrete) convexity assumption on the sequence $\{\|f\circ\Phi^{n}\|_{H^1}\}_{n\in\mathbb{Z}}.$ A small technical difficulty in our proof is that we must apply a version of the above lemma to the composition of different mappings (but that all belong to some class allowing for similar argumentation). This idea is partially inspired by our work with J. Mattingly on enhanced dissipation and mixing for (time-periodic) alternating shear flows \cite{ELM} and the previous work \cite{CEIM}.

The velocity fields we construct here are all alternating shear flows consisting of a succession of sawtooth shear flows of rapidly increasing frequency and rapidly decaying amplitude defined on a time interval $[0,T]$. There are thus three parameters that define our flows: the amplitudes $\alpha_j$, frequencies $N_j$, and the runtimes $t_j$ of the successive shear flows. The full Lagrangian flow-map associated to the velocity fields can thus be written as a composition of piece-wise linear maps (each of which depends on the triple $(\alpha_j, N_j, t_j))$:
\[\Phi(t)=\mathcal{U}_1\circ\mathcal{U}_2\circ...\circ\mathcal{U}_{N-1}\circ\mathcal{U}_{N}(t),\] for any $t\in [0,T).$
For piece-wise linear maps, it turns out that the assumption \eqref{fwdbckwd} (and its multiple map variant, given in Lemma \ref{lem:fwdbwd}) can be checked simply by computing the singular values of a $4\times 2$ constant matrix. From there, we argue relatively softly that the expected growth of the $H^1$ norm of solutions occurs for all smooth initial data \emph{at a universal rate}. This universality as well as the growth condition given in \eqref{H1Growth} imposes one condition on the triple $(\alpha_j, N_j, t_j)$, which essentially requires $N_j$ to grow sufficiently rapidly.  To show that \eqref{EnergySpectrum} holds, we observe as in \cite{DEIJ22} that this is implied by a ``balanced growth" condition. Namely, if we can prove that solutions to \eqref{eq:AE} satisfy a reverse interpolation estimate:
\[\|f\|_{H^\sigma}\leq C\frac{\|f\|_{H^1}^{\sigma}}{\|f\|_{L^2}^{\sigma-1}}\] for some fixed $\sigma>1$ and $C>0$, then \eqref{EnergySpectrum} holds automatically. Establishing this is relatively straightforward and this imposes another condition on $(\alpha_j, N_j,t_j)$. Finally, ensuring that the velocity field (and/or the scalar) satisfies the correct regularity bounds imposes a final condition on the parameters.  

\subsection{Sobolev space and Fourier analysis conventions} \label{sec:Fourier}

Some of the exact constants are important in the proofs (e.g., the fact that the constant prefactor in the first term on the right-hand side of \eqref{eq:upper2} is exactly one), and so before proceeding we define precisely the Sobolev norms that we are using. We identify $\T^2$ with $[-\pi,\pi)^2$. For $f \in L^2(\T^2)$, we define its Fourier series $\hat{f}:\Z^2 \to \C$ by 
$$\hat{f}(k,\ell) = \frac{1}{2\pi} \int_{\T^2} e^{-i(kx + \ell y)}f(x,y)\dx\dy.$$
Then, $f$ is recovered by the Fourier inversion formula 
$$ f(x,y) = \frac{1}{2\pi} \sum_{k,\ell \in \Z} e^{i(kx + \ell y)} \hat{f}(k,\ell) $$
and with our normalization conventions Plancherel's theorem reads 
$$ \|f\|_{L^2}^2 = \sum_{k, \ell \in \Z^2}|\hat{f}(k,\ell)|^2. $$
For $\sigma \ge 0$, the Sobolev space $H^\sigma$ is defined by 
$$H^\sigma = \left\{f \in L^2: \|f\|_{H^\sigma} < \infty\right\}, \quad \|f\|_{H^\sigma}: = \left(\sum_{(k,\ell)\in \Z^2} (1+|k|^2 + |\ell|^2)^\sigma|\hat{f}(k,\ell)|^2\right)^{1/2}$$
and for $f \in H^\sigma$, the homogeneous Sobolev seminorm is defined by 
$$\|f\|_{\dot{H}^\sigma}: = \left(\sum_{(k,\ell)\in \Z^2\setminus (0,0)} (|k|^2 + |\ell|^2)^\sigma|\hat{f}(k,\ell)|^2\right)^{1/2}.$$
For $s \ge 0$ we write $D^s = (-\Delta)^{s/2}$. That is, $D^s$ is the Fourier multiplier with symbol $(|k|^2 + |\ell|^2)^{s/2}$. We also define $D_x^s = (-\partial_{xx})^{s/2}$ and $D_y^s=(-\partial_{yy})^{s/2}$ to be the Fourier multipliers with symbols $|k|^s$ and $|\ell|^s$, respectively. In Appendix~\ref{appendix}, we recall some Sobolev interpolation and commutator inequalities that will be required in the proof.

\section{Proof of main theorems}
 
The proof of Theorem~\ref{thrm:main} is based on constructing a velocity field that satisfies the hypotheses of the abstract criterion for anomalous dissipation given in \cite[Proposition 1.3]{DEIJ22} for every smooth initial data. We begin in Section~\ref{sec:criteria} by recalling a version of this criterion, Proposition~\ref{prop:ADcriteria} below, that is suitable for our setting. Then, in Section~\ref{sec:udefinition} we define the velocity field used to prove Theorem~\ref{thrm:main}. The bulk of the paper consists of Sections
 ~\ref{sec:lowerbounds} and~\ref{sec:upperbounds}. Here, we prove the upper and lower bounds on the growth of Sobolev norms for the solution of \eqref{eq:AE} needed to apply Proposition~\ref{prop:ADcriteria}. Finally, in Section~\ref{sec:finish} we conclude the proof of Theorem~\ref{thrm:main} and then in Section~\ref{sec:thrm2} make the appropriate modifications to prove Theorem~\ref{thrm:2}. 

\subsection{Criteria for anomalous dissipation} \label{sec:criteria}
We begin with a criterion for anomalous dissipation which is a modified version of \cite[Corollary 1.5]{DEIJ22} with $H^2$ replaced by $H^\sigma$ for some $\sigma \in (1,2]$. The proof is exactly the same as in \cite{DEIJ22} after noting that the balanced growth condition of \cite[Lemma 1.4]{DEIJ22} holds just as well with the reverse interpolation 
$$ \|f\|_{L^2}\|f\|_{\dot{H}^2} \le C\|\dot{f}\|_{H^1}^2 $$
for some $C \ge 1$ replaced by 
$$ \|f\|_{L^2}^{\sigma-1}\|f\|_{\dot{H}^\sigma} \le C\|f\|_{\dot{H}^1}^\sigma $$
for any $\sigma > 1$. We use a criterion that allows for fractional Sobolev regularity because it is most convenient in the proof of norm growth to use a velocity field which is only $H^{3/2-}$. 

\begin{proposition} \label{prop:ADcriteria}
    Fix $T > 0$, $\sigma \in (1,2]$, and let $u \in L_{\text{loc}}^\infty([0,T);W^{1,\infty}(\T^2))$ be a divergence free velocity field. Let $f_0 \in H^\sigma$ be a mean-zero initial data and suppose that there exists $C > 1$ such that the solution to the transport equation \eqref{eq:AE}
satisfies the following two hypotheses:
\begin{enumerate}
    \item $\int_0^T \|\grad f(t)\|_{L^2}^2 \dt = \infty$,
    \item $\|f(t)\|_{L^2}^{\sigma - 1}\|f(t)\|_{\dot{H}^\sigma} \le C\|f(t)\|_{\dot{H}^1}^\sigma$ for every $t \in [0,T)$.
\end{enumerate}
Then, for every $\kappa \in (0,1)$ the solution of \eqref{eq:ADE} with the same initial data $f_0$ satisfies 
$$\kappa\int_0^T \|\grad f^\kappa(t)\|_{L^2}^2 \dt \ge \chi \|f_0\|_{L^2}^2, \quad \text{where} \quad \chi = \frac{1}{16}\left(\frac{1}{1+C^{\frac{1}{\sigma-1}}}\right)^{\frac{2\sigma}{\sigma-1}}. $$
\end{proposition}

\subsection{Construction and regularity of the velocity field} \label{sec:udefinition}

For a particular $T_* > 0$, we now define the divergence-free velocity field $u:[0,T_*] \times \T^2 \to \R^2$ that we will use to prove Theorem~\ref{thrm:main}. The fact that this is sufficient to prove the result for general $T > 0$ follows from a simple scaling argument by defining $\tilde{u}(t) = (T_*/T) u(T_*t/T)$.

\subsubsection{Definition of $u:[0,T_*]\times \T^2 \to \R^2$}
Recall that we identify $\T^2$ with $[-\pi,\pi)^2$. Define $S:\T \to \R$ by
\begin{equation}
S(x) = |x|
.
\end{equation}
For parameters $\alpha, N \in \R$ let $H_{\alpha,N}$ and $V_{\alpha,N}$ denote the shear flows
\begin{equation}
H_{\alpha,N}(x,y) =
\begin{pmatrix}
 \alpha S(Ny) \\ 0
\end{pmatrix}, \quad V_{\alpha,N}(x,y) = \begin{pmatrix}
    0 \\ \alpha S(Nx)
\end{pmatrix}.
\end{equation}
Our chosen velocity field $u$ will alternate in time along a sequence of decreasing time steps between $H_{\alpha,N}$ and $V_{\alpha,N}$ for appropriately chosen parameters $\alpha$ and $N$ that vary at each step. Below we will specify a suitable sequence of frequencies $\{N_j\}_{j=1}^\infty$, amplitudes $\{\alpha_j\}_{j=1}^\infty$, and time steps $\{t_j\}_{j=1}^\infty$ with $\sum_{j=1}^\infty t_j < \infty$. For the time steps $\{t_j\}_{j=1}^\infty$ to be defined, let $T_0 = 0$ and $T_j = 2\sum_{n=1}^j t_j$ for $j \in \N$. Then, the terminal time $T_* > 0$ is defined by
\begin{equation}
    T_* := 2\sum_{j=1}^\infty t_j.
\end{equation}
Let $\psi \in C_c^\infty((0,1))$ be a smooth function satisfying $\psi \ge 0$ and $\int_0^1 \psi(t) \dt = 1$. Then, we define $u:[0,T_*]\times \T^2 \to \R^2$ for $t \in [T_{j-1}, T_j)$ by 
\begin{equation}
    u(t) = 
    \begin{cases}
   \psi\left(\frac{t-T_{j-1}}{t_j}\right) H_{\alpha_j, N_j} & t\in [T_{j-1}, T_{j-1}+t_j), \\
 \psi\left(\frac{t-T_{j-1}-t_j}{t_j}\right) V_{\alpha_j, N_j} & t \in [T_{j-1}+t_j, T_j).
    \end{cases}
\end{equation}
Note that since $\int_0^1 \psi(t)\dt = 1$, the flow map associated with $u$ at the discrete times $T_j$ and $T_j + t_j$ is the same as it would be if $\psi$ were removed from the definition. The time dependence involving $\psi$ is included so that $u$ can be regular in time. Ignoring $\psi$, the schematic for how the velocity field alternates in time is
$$ \underbrace{H_{\alpha_1, N_1}, V_{\alpha_1, N_1}}_{t \in [0,T_1)}, \underbrace{H_{\alpha_2, N_2}, V_{\alpha_2, N_2}}_{t\in [T_1,T_2)}, \underbrace{H_{\alpha_3, N_3}, V_{\alpha_3, N_3}}_{t\in [T_2, T_3)}, \ldots,  $$
where each $H_{\alpha_j,N_j}$ and $V_{\alpha_j,N_j}$ runs for time $t_j$. 

\subsubsection{Choice of parameters and regularity} \label{sec:parameters}
With the construction above it is clear that $u \in L^\infty_{\text{loc}}([0,T_*); W^{1,\infty}(\T^2))$, as is required to apply Proposition~\ref{prop:ADcriteria}. Additionally, one can easily check that a sufficient condition to have $u \in C^\infty([0,T];C^\alpha(\T^2)))$ as well as the regularity claimed in \eqref{eq:velocityregularity} is
\begin{equation} \label{eq:velconditions}
    \sup_{j\in \N}\frac{\alpha_j N_j^\alpha}{t_j^m} + \sup_{j \in \N}\frac{\alpha_j N_j}{1+|\log(N_j)|^4} < \infty,
\end{equation} for all $m\in\mathbb{N}$ and $\alpha<1.$
The bound on the first term gives the time regularity, while the bound on the second term implies that $u$ possesses the modulus of continuity $\omega(s) = s(1+(\log(s))^4)$ uniformly in time. For $M \ge 2$ to be chosen sufficiently large, we choose the parameters $N_j = 2^j$, 
$$\alpha_j = 2^{-j}(1+|\log(N_j)|^4) = 2^{-j}\left(1 + j^4 |\log(2)|^4\right),$$ 
and $t_j = 2\lceil M j^{5/2}\rceil /(\alpha_j N_j)$, where $\lceil x \rceil$ denotes the first integer greater than or equal to $x$. Then, $T_* = 2\sum_{j}t_j < \infty$ and it is easy to check that \eqref{eq:velconditions} holds. For convenience of notation, we define $K_j : = \alpha_j N_j t_j = 2\lceil M j^{5/2}\rceil$.

\subsection{Norm growth} \label{sec:lowerbounds}
From here until Section~\ref{sec:thrm2}, $N_j$, $\alpha_j$, and $t_j$ denote the parameter choices of Section~\ref{sec:parameters} for some $M \ge 2$ to be chosen sufficiently large. For $j \in \N$ define the Lebesgue measure preserving homeomorphisms
$$\phi_j(x,y) = 
\begin{pmatrix}
    x + \alpha_j t_j S(N_j y) \\ 
    y
\end{pmatrix},
\quad 
\psi_j(x,y) = 
\begin{pmatrix}
    x \\ 
    y + t_j \alpha_j S(N_j x)
\end{pmatrix},
\quad 
\Phi_j = \psi_j \circ \phi_j.
$$
For a solution of \eqref{eq:AE} we will write $f_j = f(T_j)$. Note that the definitions above are such that 
\begin{equation}
    f_j = f_0 \circ \Phi_1^{-1}\circ \Phi_2^{-1} \circ \ldots \circ \Phi_j^{-1}.
\end{equation}

A crucial step in verifying both hypotheses of Proposition~\ref{prop:ADcriteria} is obtaining essentially sharp lower bounds on the exponential growth of the $H^1$ norm of solutions to \eqref{eq:AE}. Since, defining $\bar{f}_j = f_{j-1} \circ \phi_j^{-1}$, we have 
\begin{align*}
    \partial_y\bar{f}_j &= (\partial_y f_{j-1}) \circ \phi_{j}^{-1} - K_j (\partial_x f_{j-1})\circ \phi_{j}^{-1}, \\ 
    \partial_x f_j &= (\partial_x \bar{f}_j)\circ \psi_j^{-1} - K_j (\partial_y \bar{f}_j)\circ \psi_j^{-1},
\end{align*}
one expects that over the time interval $[T_{j-1},T_j]$ the $H^1$ norm of a solution is amplified by the factor $K_j^2$ (if possible cancellations can be ignored). Lemma~\ref{lem:H1growth} below shows that this is indeed the case. There is a small complication with the preceding idea due to the fact that we need to consider all $t\in [0,T_*)$ and not just $T_j.$ To deal with the growth between the discrete times $T_j$, for $j \in \N$ we define the increasing function $\zeta_j:[0,t_j]\to [0,1]$ with $\zeta_j(0) = 0$ and $\zeta_j(t_j) = 1$ by 
\begin{equation} \label{eq:zetaj}
    \zeta_j(t) = \int_0^{t/t_j}\psi(\tau)\dee \tau,
\end{equation}
where $\psi$ is as defined in Section~\ref{sec:udefinition}. Then, let 
$$ \tilde{h}_j(t) = \begin{cases}
        K_j \zeta_j(t-T_{j-1}) & t \in [T_{j-1}, T_{j-1}+t_j], \\ 
        K_j^2 \zeta_j(t-T_{j-1}-t_j) & t\in [T_{j-1}+t_j, T_{j}]
    \end{cases} 
$$
and define $h_j:[T_{j-1},T_j]\to [0,\infty)$ by 
\begin{equation} \label{eq:hj}
h_j(t) = \begin{cases} \max(\tilde{h}_j(t),1) & t \in [T_{j-1}, T_{j-1}+t_j], \\ 
\max(\tilde{h}_j(t),K_j) & t\in [T_{j-1}+t_j, T_{j}].
\end{cases} 
\end{equation}
Our main lower bound on the $H^1$ growth of solutions of \eqref{eq:AE} is then given as follows.

\begin{lemma}[$H^1$ growth] \label{lem:H1growth}
Let $M$ be as defined in Section~\ref{sec:parameters} and chosen sufficiently large. For every mean-zero $f_0 \in H^1$ there exists a constant $c$ depending only on an upper bound for $\|f_0\|_{\dot{H}^1}/\|f_0\|_{L^2}$ such that for every $j \in \N$ and $t \in [T_{j-1},T_j]$ the solution of \eqref{eq:AE} satisfies
\begin{equation} \label{eq:H1growthlemma}
    \|f(t)\|_{\dot{H}^1} \ge c h_j(t) \|f_0\|_{\dot{H}^1}\prod_{n=1}^{j-1}K_n^2,
\end{equation}
where $h_j$ is as defined above in \eqref{eq:hj}. In particular, for any mean-zero and nontrivial initial data $f_0$ we have 
\begin{equation} \label{eq:blowup}
    \int_0^{T_*}\|\grad f(t)\|_{L^2}^2 \dt \ge c\|f_0\|_{\dot{H}^1}^2 \sum_{j=1}^\infty t_j \prod_{n=1}^{j-1}K_n^4 = \infty, 
\end{equation}
where the divergence of the sum follows easily from the definitions of $\alpha_j$, $N_j$, and $t_j$.
\end{lemma}

\subsubsection{Proof of discrete-time $H^1$ growth}

We will first obtain a lower bound on the $\|f_j\|_{\dot{H}^1}$ and then upgrade to continuous time. The key lemma needed to prove the growth at discrete times, which one should view as a generalization of \eqref{fwdbckwd}, is the following.
\begin{lemma} \label{lem:fwdbwd}
There exist constants $c,C > 0$ so that if $M$ is sufficiently large then for every mean-zero $g \in H^1$ and $j \in \N$ we have the estimates
\begin{align}
\|g \circ \Phi_{j+1}^{-1}\|_{\dot{H}^1}^2 + \|g \circ \Phi_{j}\|_{\dot{H}^1}^2 &\ge cK_j^4\|g\|_{\dot{H}^1}^2 \label{eq:FBdiff} \\ 
\|g \circ \Phi_{j+2}^{-1}\|_{\dot{H}^1}^2 + \|g \circ \Phi_{j+2}\|_{\dot{H}^1}^2 &\ge (K_{j+2}^4 - CK_{j+2}^3)\|g\|_{\dot{H}^1}^2. \label{eq:FBsame}
\end{align}
\end{lemma}

\begin{proof}
    Both of the estimates require us to bound from below 
$$ \|\grad(g \circ \Phi_{k})\|_{L^2}^2 + \|\grad(g\circ \Phi_{\ell}^{-1})\|_{L^2}^2 $$
for some choices for $k,\ell \in \N$. By the chain rule and the fact that $\Phi_{k}$ is area preserving we have 
\begin{equation} \|\grad (g \circ \Phi_{k})\|_{L^2}^2 = \int_{\T^2} |(\grad \Phi_{k})^T (\grad g) \circ \Phi_{k}|^2 = \int_{\T^2} |A_{k} \grad g|^2,
\end{equation}
where $A_{k}: \T^2 \to \R^{2\times 2}$ is the matrix valued function given by 
\begin{equation}  A_{k} = (\grad \Phi_{k})^T \circ \Phi^{-1}_{k}.
\end{equation}
A direct computation shows that
\begin{equation}
A_{k}(x,y) = \begin{pmatrix}
        1 & K_k S'(N_k x) \\
        K_k S'(N_k(y-\alpha_k t_k S(N_k x))) & 1 + K^2_k S'(N_k x)S'(N_k(y-\alpha_k t_k S(N_k x)))
    \end{pmatrix}.
\end{equation}
A similar calculation yields 
\begin{equation}
    \|\grad(g \circ \Phi_\ell^{-1})\|_{L^2}^2 = \int_{\T^2}|B_{\ell} \grad g|^2,
\end{equation}
where 
\begin{equation}
B_{\ell}(x,y) = \begin{pmatrix}
    1 + K_\ell^2 S'(N_\ell y) S'(N_\ell(x+\alpha_\ell t_\ell S(N_\ell y))) & -K_\ell S'(N_\ell(x+\alpha_\ell t_\ell S(N_\ell y))) \\ 
    -K_\ell S'(N_\ell y) & 1
\end{pmatrix}.
\end{equation}
Define $Q_{k,\ell}:\T^2 \to \R^{4\times 2}$ by 
\begin{equation}
    Q_{k,\ell}(x,y) = \begin{pmatrix}
        A_{k}(x,y) \\ 
        B_{\ell}(x,y)
    \end{pmatrix}.
\end{equation}
Then, 
\begin{equation} \label{eq:H1sum}
    \|g\circ \Phi_{k}\|_{\dot{H}^1}^2 + \|g \circ \Phi_{\ell}^{-1}\|_{\dot{H}^1}^2 = \int_{\T^2}|Q_{k,\ell}\grad g|^2 = \int_{\T^2}(Q_{k,\ell}^T Q_{k,\ell} \grad g, \grad g)
\end{equation}
and to prove \eqref{eq:FBdiff} and \eqref{eq:FBsame} it suffices to suitably bound from below 
$$ Q_{k,\ell}^T(x,y)^T Q_{k,\ell}(x,y) v \cdot v $$
for general $v \in \R^2$ in the cases $(k,\ell) = (j,j+1)$ and $(k,\ell) = (j,j)$ uniformly on the full measure set where the derivatives in $A_{k}$ and $B_{\ell}$ are all defined. For any $(x,y) \in \T^2$ where all of the derivatives are defined, let 
\begin{align}
    a_1 &= S'(N_k x), \\
    a_2 &= S'(N_k(y-\alpha_k t_k S(N_kx))), \\ 
    a_3 &= S'(N_\ell y), \\
    a_4 & = S'(N_\ell(x+\alpha_\ell t_\ell S(N_\ell y))). \\
\end{align}
Then, $a_j \in \{1,-1\}$ and we have 
\begin{equation}
    Q_{k,\ell} = \begin{pmatrix}
        1 & K_k a_1 \\
        K_k a_2 & 1 + K^2_k a_1 a_2 \\
        1 + K_\ell^2 a_3 a_4 & -K_\ell a_4 \\
        -K_\ell a_3 & 1
    \end{pmatrix}.
\end{equation}
We first let $(k,\ell) = (j,j)$ for some $j \in \N$ and prove \eqref{eq:FBsame}. In this case it is straightforward to check that there exists a matrix $P_1 \in \R^{2\times 2}$ with each entry bounded by a constant $C_1$ that does not depend on $j$ such that 
$$ Q_{j,j}^T Q_{j,j} = K_j^4\left(\mathbf{1} + \frac{P_1}{K_j}\right).$$
This implies that 
$$ Q_{j,j}^T Q_{j,j} v \cdot v \ge (K_j^4 - 2C_1K_j^3)|v|^2 $$
for every $v \in \R^2$, which together with \eqref{eq:H1sum} implies \eqref{eq:FBsame}. For \eqref{eq:FBdiff}, we consider the case $(k,\ell) = (j,j+1)$ and compute that 
$$ Q_{j,j+1}^T Q_{j,j+1} = \begin{pmatrix} 
K_{j+1}^4 & 0 \\ 
0 & K_j^4
\end{pmatrix} 
+ (K_j^3 + K_{j+1}^3) P_2
$$
for a matrix $P_2 \in \R^{2\times 2}$ with again entries uniformly bounded by a constant $C_2$ that does not depend on $j$. Since $1 \le K_{j+1}/K_j \le 17$ for any $j$, it follows then that for any $v \in \R^2$ we have 
\begin{equation} \label{eq:jjplus1}
Q_{j,j+1} Q_{j,j+1}^T v \cdot v \ge (K_j^4 - 2(1+17^3)C_2K_j^3)|v|^2 \ge \frac{1}{2}K_j^4 |v|^2, 
\end{equation}
where the second inequality follows from $K_j \ge M$ by choosing $M$ sufficiently large. Combining \eqref{eq:jjplus1} with \eqref{eq:H1sum} proves \eqref{eq:FBdiff}.
\end{proof}

We can now prove the sharp $H^1$ growth at the discrete times $T_j$ by using an argument based on the idea of Lemma~\ref{fwdbckwdp} discussed in the introduction.

\begin{lemma}\label{lem:discretegrowth}
Let $M$ be as defined in Section~\ref{sec:parameters} and chosen sufficiently large. For every mean-zero $f_0 \in H^1$ there exists a constant $c$ depending only on an upper bound for $\|f_0\|_{\dot{H}^1}/\|f_0\|_{L^2}$ such that for every $j \in \N$ the solution of \eqref{eq:AE} satisfies
\begin{equation} 
    \|f_j\|_{\dot{H}^1} \ge c\|f_0\|_{\dot{H}^1}\prod_{n=1}^{j}K_n^2.
\end{equation}
\end{lemma}

\begin{proof}
By linearity and the fact that \eqref{eq:AE} conserves the $L^2$ norm we may assume without loss of generality that $\|f_0\|_{L^2} = 1$. Throughout this proof, $c$ and $C$ denote the constants from Lemma~\ref{lem:fwdbwd}.

\vspace{0.2cm}
\noindent \textbf{Claim 1}: There exists $J \in \N$ depending only on $\|f_0\|_{\dot{H}^1}$ such that for some $j_0 \in \N \cup \{0\}$ with $j_0 \le J$ there holds
\begin{equation} \label{eq:growth1}
   \|f_{j_0+1}\|_{\dot{H}^1} \ge \|f_{j_0}\|_{\dot{H}^1}.
\end{equation}
In fact, one can take $J = \lceil \log(\|f_0\|_{\dot{H}^1})\rceil$.
\begin{proof}[Proof of Claim 1]
Let $j \in \N$ be such that $\|f_{j+1}\|_{\dot{H}^1} \le \|f_j\|$. By \eqref{eq:FBdiff} applied with $g = f_j$ we have 
\begin{equation} \label{eq:growth2}
\|f_{j+1}\|_{\dot{H}^1}^2 + \|f_{j-1}\|_{\dot{H}^1}^2 \ge c K_j^4 \|f_j\|_{\dot{H}^1}^2.
\end{equation}
Since $\|f_{j+1}\|_{\dot{H}^1} \le \|f_j\|$, assuming $M$ is large enough so that $cK_j^4/2 \ge 1$ for all $j$, it follows that
\begin{equation} \|f_j\|_{\dot{H}^1}^2 \le \frac{2}{c K_j^4}\|f_{j-1}\|_{\dot{H}^1}^2. 
\end{equation}
Therefore, if \eqref{eq:growth1} fails for all $0 \le j_0 \le J$ and $M$ is large enough so that $2/(cK_j^4) \le e^{-2}$ for every $j$, we have 
\begin{equation}
    \|f_J\|_{\dot{H}^1}^2 \le \|f_0\|_{\dot{H}^1}^2 \prod_{j=1}^J \frac{2}{cK_j^4} \le e^{-2J}\|f_0\|_{\dot{H}^1}^2.
\end{equation}
It follows then by the Poincar\'{e} inequality that 
\begin{equation} 
1 = \|f_0\|_{L^2}^2 = \|f_J\|_{L^2}^2 \le \|f_J\|_{\dot{H}^1}^2 \le e^{-2J}\|f_0\|_{\dot{H}^1}^2,
\end{equation}
and so 
\begin{equation}\label{eq:growth4}
    J \le \log(\|f_0\|_{\dot{H}^1}).
\end{equation}
\end{proof}

\noindent \textbf{Claim 2}: There exists a constant $C_1 > 0$ so that if $\|f_{j}\|_{\dot{H}^1} \ge \|f_{j-1}\|_{\dot{H}^1}$ for some $j \in \N$, then 
\begin{align} 
    \|f_{j+1}\|_{\dot{H}^1} &\ge \|f_{j}\|_{\dot{H}^1}, \label{eq:growth4}\\
    \|f_{j+2}\|_{\dot{H}^1}^2 &\ge (K_{j+2}^4 - C_1K_{j+2}^3)\|f_{j+1}\|_{\dot{H}^1}^2. \label{eq:growth5}
\end{align}

\begin{proof}[Proof of Claim 2]
By \eqref{eq:FBdiff} applied with $g = f_j$ we have 
\begin{equation}
    \|f_{j-1}\|_{\dot{H}^1}^2 + \|f_{j+1}\|_{\dot{H}^2}^2 \ge c K_j^4 \|f_j\|_{\dot{H}^1}.
\end{equation}
Since $\|f_j\|_{\dot{H}^1} \ge \|f_{j-1}\|_{\dot{H}^1}$ it follows that 
\begin{equation} \label{eq:growth6}
   \|f_j\|_{\dot{H}^1}^2 \le \frac{2}{cK_j^4} \|f_{j+1}\|_{\dot{H}^1}^2,
\end{equation}
which implies \eqref{eq:growth4} and will also be crucial to obtaining \eqref{eq:growth5}. To prove \eqref{eq:growth5} we begin by applying \eqref{eq:FBsame} with $g = f_{j+1}$ to get
\begin{equation} \label{eq:growth10}
    \|f_{j+2}\|_{\dot{H}^1}^2 + \|f_{j+1}\circ \Phi_{j+2}\|_{\dot{H}^1}^2 \ge (K_{j+2}^4 - CK_{j+2}^3)\|f_{j+1}\|_{\dot{H}^1}^2.
\end{equation}
Observe now that 
\begin{equation}
    f_{j+1} \circ \Phi_{j+2} = f_j \circ (\Phi_{j+1}^{-1} \circ \Phi_{j+2}) = f_j \circ (\phi_{j+1}^{-1} \circ \psi_{j+1}^{-1} \circ \psi_{j+2} \circ \phi_{j+2})
\end{equation}
and so by the fact that 
$$\grad(\psi_{j+1}^{-1} \circ \psi_{j+2})(x,y) = \begin{pmatrix}
    1 & 0 \\ 
    K_{j+2}S'(N_{j+2}x) - K_{j+1}S'(N_{j+1}x) & 1
\end{pmatrix} $$ 
we have 
\begin{equation}
\|\grad (\phi_{j+1}^{-1} \circ \psi_{j+1}^{-1} \circ \psi_{j+2} \circ \phi_{j+2})\|_{L^\infty} \le \|\grad \phi_{j+1}^{-1}\|_{L^\infty} \|\grad(\psi_{j+1}^{-1} \circ \psi_{j+2})\|_{L^\infty} \|\grad \phi_{j+2}\|_{L^\infty} \le C_2K_{j+2}^3
\end{equation}
for some constant $C_2$ that does not depend on $j$. Employing \eqref{eq:growth6} we deduce that there is $C_3$ independent of $j$ such that 
\begin{equation} \label{eq:growth9}
\|f_{j+1} \circ \Phi_{j+2}\|_{\dot{H}^1}^2 \le C_2^2 K_{j+2}^6\|f_{j}\|_{\dot{H}^1}^2 \le C_3K_{j+2}^2 \|f_{j+1}\|_{\dot{H}^1}^2. 
\end{equation}
Putting \eqref{eq:growth9} into \eqref{eq:growth10} completes the proof of \eqref{eq:growth5}.
\end{proof}

We are now ready to complete the proof of the lemma. By Claim 1, there exists $j_0 \in \N$ with $j_0 \le \lceil\log(\|f_0\|_{\dot{H}^1})\rceil + 1$ such that 
\begin{equation}
    \|f_{j_0}\|_{\dot{H}^1} \ge \|f_{j_0-1}\|_{\dot{H}^1}.
\end{equation}
Iterating \eqref{eq:growth4} of Claim 2 it follows that $\|f_{j+1}\|_{\dot{H}^1} \ge \|f_j\|_{\dot{H}^1}$ for every $j \ge j_0$. Therefore, by Claim 2, \eqref{eq:growth5} holds for every $j \ge j_0$. Given $n \ge j_0+2$ we iterate \eqref{eq:growth5} over $j_0 \le j \le n$ and use also $\|f_{j_0+1}\|_{\dot{H}^1} \ge \|f_{j_0}\|_{\dot{H}^1}$ to conclude 
\begin{equation}
    \|f_n\|_{\dot{H}^1}^2 \ge \|f_{j_0}\|_{\dot{H}^1}^2 \prod_{j={j_0+2}}^n (K_j^4 - C_1 K_j^3)\ge c_1 \|f_0\|_{\dot{H}^1}^2 \prod_{j=1}^n (K_j^4 - C_1 K_j^3),
\end{equation}
where 
$$c_1 = \frac{1}{\|f_0\|_{\dot{H}^1}^2} \left(\prod_{j=1}^{\lceil \log(\|f_0\|_{\dot{H}^1})\rceil + 2}K_j^4\right)^{-1}.  $$
The result then follows from 
$$ \prod_{j=1}^n(K_j^4 - C_1 K_j^3) = \left(\prod_{j=1}^n K_j^4\right)\left(\prod_{j=1}^n(1-C_1K_j^{-1})\right)$$
and the fact that 
$$ \prod_{j=1}^\infty(1-C_1K_j^{-1}) > 0$$
because $\sum_{j=1}^\infty K_j^{-1} < \infty$ due to $K_j \ge j^{5/2}$.
\end{proof}

\subsubsection{Proof of continuous time $H^1$ growth}
We now upgrade Lemma~\ref{lem:discretegrowth} to continuous time and complete the proof of Lemma~\ref{lem:H1growth}.

\begin{proof}[Proof of Lemma~\ref{lem:H1growth}]
As before, we may assume without loss of generality that $\|f_0\|_{L^2} = 1$. Moreover, it is sufficient to prove the estimate for all $j \ge j_0$ with $j_0$ depending only on an upper bound for $\|f_0\|_{\dot{H}^1}$. Throughout this proof, $c > 0$ denotes the constant from Lemma~\ref{lem:discretegrowth}.

We begin by using Lemma~\ref{lem:discretegrowth} to obtain lower bounds on specific derivatives of $f_j$ as well as the solution at the intermediate time $f(T_j + t_{j+1}) = f_j \circ \phi_{j+1}^{-1}$. A straightforward computation with the chain rule shows that for any $g \in H^1$ and $j \in \N$ we have
\begin{equation} \label{eq:trivialchain}
\|\grad (g\circ \phi_j^{-1})\|_{L^2} \le (K_j+2)\|\grad g\|_{L^2}  
\end{equation}
as well as the same bound with $\phi_j^{-1}$ replaced by $\psi_j^{-1}$. Thus,
$$ \|\grad (g \circ \Phi_j^{-1})\|_{L^2} \le (K_j+2)^2\|\grad g\|_{L^2} \le (K_j^2 + 5K_j)\|\grad g\|_{L^2}. $$
Iterating this estimate and using $\sum_{j=1}^\infty K_j^{-1} < \infty$ we see that there is $C_1$ independent of $j$ such that 
\begin{align} 
\|f_j\|_{\dot{H}^1} & \le C_1 \|f_0\|_{\dot{H}^1} \prod_{n=1}^j K_n^2, \label{eq:H1ubd1}\\
\|f_j \circ \phi_{j+1}^{-1}\|_{\dot{H}^1} &\le C_1 \|f_0\|_{\dot{H}^1} K_{j+1} \prod_{n=1}^j K_n^2. \label{eq:H1ubd2}
\end{align}
For simplicity of notation, let $\bar{f}_{j+1} = f_j \circ \phi_{j+1}^{-1}$. Since
\begin{equation}
  \grad f_{j+1} (x,y) = 
  \begin{pmatrix}
  (\partial_x \bar{f}_{j+1}) \circ \psi_{j+1}^{-1}(x,y) - K_{j+1}S'(N_{j+1}x)(\partial_y \bar{f}_{j+1})\circ \psi_{j+1}^{-1}(x,y) \\ 
  (\partial_y \bar{f}_{j+1})\circ \psi_{j+1}^{-1}(x,y)
  \end{pmatrix},
\end{equation}
it follows from Lemma~\ref{lem:discretegrowth} and \eqref{eq:H1ubd2} that
$$c \|f_0\|_{\dot{H}^1} K_{j+1}^2\prod_{n=1}^{j} K_n^2 \le  \|\grad f_{j+1}\|_{L^2} \le K_{j+1}\|\partial_y \bar{f}_{j+1}\|_{L^2} + 2C_1\|f_0\|_{\dot{H}^1}K_{j+1}\prod_{n=1}^j K_n^2. $$
Thus, if $j$ is large enough so that $cK_{j+1}^2/2 \ge 2C_1 K_{j+1}$, then 
\begin{equation} \label{eq:partialybound}
    \|\partial_y \bar{f}_{j+1}\|_{L^2} \ge \frac{c}{2}\|f_0\|_{\dot{H}^1}K_{j+1}\prod_{n=1}^j K_n^2.
\end{equation}
Note that since $c$ depends only on an upper bound for $\|f_0\|_{\dot{H}^1}$, so does this choice of $j$. A similar argument using \eqref{eq:partialybound} and \eqref{eq:H1ubd1} shows that for $j$ sufficiently large we also have 
\begin{equation}\label{eq:partialxbound}
    \|\partial_x f_j\|_{L^2} \ge \frac{c}{4}\|f_0\|_{\dot{H}^1} \prod_{n=1}^j K_n^2.
\end{equation}

We are now ready to complete the proof. Let $j_0$ be large enough so that both \eqref{eq:partialybound} and \eqref{eq:partialxbound} hold for all $j \ge j_0 - 1$. Fix $j \in \N$ with $j \ge j_0$ and $t \in [T_{j-1},T_{j-1}+t_j]$. Then, defining $\phi_{j,t}(x,y) = (x+\alpha_j t_j \zeta_j(t-T_{j-1})S(N_j y), y)$, for such $t$ we have
$$ \grad f(t,x,y) = 
\begin{pmatrix}
    (\partial_x f_{j-1}) \circ \phi_{j,t}^{-1}(x,y) \\ 
    (\partial_y f_{j-1})\circ \phi_{j,t}^{-1}(x,y) - K_j \zeta_j(t-T_{j-1})S'(N_j y) (\partial_x f_{j-1}) \circ \phi_{j,t}^{-1}(x,y)
\end{pmatrix},
$$
where $\zeta_j:[0,t_j]\to [0,1]$ is as defined in \eqref{eq:zetaj}. Thus, by \eqref{eq:partialxbound} there holds
\begin{equation}\label{eq:continuoustime1}
    \|\grad f(t)\|_{L^2} \ge \frac{c}{4}\|f_0\|_{\dot{H}^1}\prod_{n=1}^{j-1} K_n^2
\end{equation}
for all $t \in [T_{j-1},T_{j-1}+t_j]$. On the other hand, if $t$ is such that $K_j \zeta_j(t-T_{j-1}) \ge 8C_1/c$, then by \eqref{eq:H1ubd1} and \eqref{eq:partialxbound} we have
\begin{align} 
\|\grad f(t)\|_{L^2} &\ge K_j \zeta_j(t-T_{j-1})\|\partial_x f_{j-1}\|_{L^2} - \|f_{j-1}\|_{\dot{H}^1} \\
&\ge \frac{1}{2}K_j \zeta_j(t-T_{j-1})\|\partial_x f_{j-1}\|_{L^2} \\ 
& \ge \frac{c}{8}K_j \zeta_j(t-T_{j-1})\|f_0\|_{\dot{H}^1}\prod_{n=1}^{j-1} K_n^2. \label{eq:continuoustime2}
\end{align}
Combining \eqref{eq:continuoustime1} and \eqref{eq:continuoustime2} proves \eqref{eq:H1growthlemma} for $t \in [T_{j-1},T_{j-1}+t_j]$. The estimate on the other half of the time interval $[T_{j-1},T_j]$ follows in a similar way using \eqref{eq:H1ubd2} and \eqref{eq:partialybound}.

\end{proof}

\subsection{Balanced growth} \label{sec:upperbounds}

Lemma~\ref{lem:H1growth} establishes the first hypothesis of Proposition~\ref{prop:ADcriteria}, and moreover shows that in order to obtain the balanced growth hypothesis we need to prove that for some $s \in (0,1]$ the $\dot{H}^{1+s}$ norm of $\|f_j\|_{\dot{H}^{1+s}}$ grows at most like $\prod_{n=1}^{j} K_n^{2s+2}$ with $j$. This is the content of the next lemma. 

\begin{lemma}[Upper bound in $H^{\sigma}$, $\sigma > 1$]\label{lem:uppermain}
    Fix $s \in (2/5,1/2)$ and let $h_j$ be as in the statement of Lemma~\ref{lem:H1growth}. For every mean-zero $f_0 \in H^{1+s} \cap W^{1,\infty}$
    there exists a constant $C > 0$ depending only on $s$ and an upper bound for $\|f_0\|_{W^{1,\infty}}/\|f_0\|_{L^2}$ such that for every $j \in \N$ and $t \in [T_{j-1},T_j]$ the solution of \eqref{eq:AE} satisfies
\begin{equation}
    \|f(t)\|_{\dot{H}^{1+s}} \le C h_j^{1+s}(t) \|f_0\|_{\dot{H}^{1+s}}\prod_{n=1}^{j-1}K_n^{2+2s}.
\end{equation}
\end{lemma}

Unlike in proof of Lemma~\ref{lem:H1growth}, the extension from discrete to continuous time is immediate and will only serve to complicate the notation. Thus, for simplicity we will prove the bound only for $t \in \{T_j\}_{j=1}^\infty$. In particular, in the notation of Lemma~\ref{lem:uppermain}, we prove in this section that there is a constant $C$ depending only an upper bound for $\|f_0\|_{\dot{H}^1}/\|f_0\|_{L^2}$ such that 
\begin{equation} \label{eq:upperdiscrete}
    \|f_j\|_{\dot{H}^{1+s}} \le \exp\left(C\left(1+\frac{\|f_0\|_{W^{1,\infty}}}{\|f_0\|_{\dot{H}^1}}\right)\right)\|f_0\|_{\dot{H}^{1+s}}\prod_{n=1}^{j}K_n^{2+2s}
\end{equation}
for every $j \in \N$.

\subsubsection{Proof of discrete-time upper bound}

We will bound $\|f_j\|_{\dot{H}^{1+s}}$ by computing $\grad(f_{j-1}\circ \Phi_j^{-1})$ and then estimating the $\dot{H}^s$ norm of the result. We thus begin with a bound for $\|f \circ \Phi_j^{-1}\|_{\dot{H}^s}$ that follows easily from interpolation theory.

\begin{lemma} \label{lem:interp1}
	Let $s \in (0,1/2)$. For every mean-zero $f \in H^{s}$ we have
	\begin{align} 
 \|(f \circ \phi^{-1}_j)\|_{\dot{H}^s} \le (K_j+2)^{s} \|f\|_{\dot{H}^s}.
 \end{align}
 Moreover, the same estimate holds with $\phi_j^{-1}$ replaced with $\psi_j^{-1}$.
\end{lemma} 

\begin{proof}
Since $\|f\circ \phi_{j}^{-1}\|_{L^2} = \|f\|_{L^2}$ the result follows immediately from \eqref{eq:trivialchain} and the interpolation theorem Lemma~\ref{lem:interp}.
\end{proof}

Now we use Lemma~\ref{lem:interp1} and the commutator estimate of Lemma~\ref{lem:Kato} to bound $\|f \circ \phi_{j}^{-1}\|_{\dot{H}^{1+s}}$.

\begin{lemma} \label{lem:upper2}
	For any $s \in (0,1/2)$ there exists a constant $C$ depending only on $s$ such that for every $f \in \dot{H}^{s+1} \cap W^{1,\infty}$ there holds 
\begin{equation} \label{eq:upper2}
\|f \circ \phi_j^{-1}\|_{\dot{H}^{s+1}} \le K_j^{s+1}\|f\|_{\dot{H}^{s+1}} + C(K_j \|f\|_{\dot{H}^{s+1}}+ K_j N_j^s \|f\|_{W^{1,\infty}}). 
\end{equation}
Moreover, the same estimate holds with $\phi_j^{-1}$ replaced by $\psi_j^{-1}$.
\end{lemma}

\begin{proof}
Recall the notation and conventions of Section~\ref{sec:Fourier}. First, note that by the definition of $\|\cdot\|_{\dot{H}^{s+1}}$ and the triangle inequality, for any $g \in \dot{H}^{s+1}$ we have
$$ \|g\|_{\dot{H}^{s+1}} = \|D^s \grad g\|_{L^2} \le \|D^s \partial_x g\|_{L^2} + 
\|D^s \partial_y g\|_{L^2}.$$
We will estimate each term for $g = f\circ \phi_j^{-1}$. For the first term, we note that since $\partial_x (f\circ \phi_j^{-1}) = (\partial_x f)\circ \phi_j^{-1}$ it follows from Lemma~\ref{lem:interp1} that there is a constant $C_1$ such that
$$\|D^s \partial_x (f\circ \phi_j^{-1})\|_{L^2} \le (K_j+2)^s\|\partial_x f\|_{\dot{H}^s} \le C_1(K_j+2)^s\|f\|_{\dot{H}^{1+s}}. $$
For the second term, we begin by computing 
$$ \partial_y (f\circ \phi_j^{-1}) = -K_j S'(N_j y)(\partial_x f)\circ \phi_j^{-1} + (\partial_y f)\circ \phi_j^{-1}. $$
Applying $D^s$ and introducing a commutator term we have
\begin{equation}\label{eq:fourterms}
\begin{aligned}
D^s \partial_y (f\circ \phi_j^{-1}) &= -K_jS'(N_j y) D^s ((\partial_x f)\circ \phi_j^{-1}) \\ 
& \quad+ [K_j S'(N_j y) D^s((\partial_x f)\circ \phi_j^{-1}) - D^s(K_j S'(N_j y)(\partial_x f) \circ \phi_j^{-1})] \\  
& \quad + D^s ((\partial_y f)\circ \phi_j^{-1}).
\end{aligned}
\end{equation}
By Lemma~\ref{lem:interp1}, we have 
$$ \|K_j S'(N_j \cdot) D^s ((\partial_x f)\circ \phi_j^{-1})\|_{L^2} = K_j \|(\partial_x f) \circ \phi_j^{-1}\|_{\dot{H}^s} \le K_j (K_j + 2)^s \|f\|_{\dot{H}^{s+1}}\le (K_j^{s+1} + 2^s K_j)\|f\|_{\dot{H}^{s+1}}.$$
For the commutator term, we use the homogeneous Kato-Ponce inequality given in Lemma~\ref{lem:Kato} to obtain
\begin{align}
\|K_j S'(N_j \cdot) D^s((\partial_x f)\circ \phi_j^{-1}) - D^s(K_j S'(N_j \cdot)(\partial_x f) \circ \phi_j^{-1})\|_{L^2} \le C_2 K_j \|D_y^s S'(N_j \cdot)\|_{L^2}\|(\partial_x f)\circ \phi_j^{-1}\|_{L^\infty}
\end{align}
for some constant $C_2$ depending on $s$.
Using that $S' \in \dot{H}^s(\T)$ since $s < 1/2$ and the fact that 
$$ \widehat{S'(N_j\cdot)}(\ell) = \begin{cases}
    \widehat{S'}(\ell/N_j) & \ell/N_j \in \Z \\ 
    0 & \ell/N_j \not \in \Z
\end{cases} $$
because $N_j$ is an integer, one can show that 
$\|D_y^s S'(N_j \cdot)\|_{L^2} \le C_3 N_j^s$ for $C_3 > 0$ depending on $s$. Thus,
$$ \|K_j S'(N_j \cdot) D^s((\partial_x f)\circ \phi_j^{-1}) - D^s(K_j S'(N_j \cdot)(\partial_x f) \circ \phi_j^{-1})\|_{L^2} \le C_2 C_3 K_j N_j^s \|f\|_{W^{1,\infty}}. $$
Bounding the final term in \eqref{eq:fourterms} using Lemma~\ref{lem:interp1} and combining all of our other estimates completes the proof of the estimate of $\|f\circ \phi_j^{-1}\|_{\dot{H}^{s+1}}$. The estimate with $\psi_j^{-1}$ is obtained in the same way by reversing the roles of $x$ and $y$.
\end{proof}

\begin{proof}[Proof of \eqref{eq:upperdiscrete}]
Applying Lemma~\ref{lem:upper2} twice we deduce that there exists a constant $C_1$ depending only on $s$ such that for every mean-zero $f_0 \in H^{1+s} \cap W^{1,\infty}$ and $j \in \N$ we have
\begin{equation}\label{eq:fullcomp}
    \|f_{j}\|_{\dot{H}^{s+1}} \le K_j^{2s+2}\|f_{j-1}\|_{\dot{H}^{s+1}}\left(1 + C_1K_j^{-s} + C_1K_j^{-s}N_j^s \frac{\|f_{j-1}\|_{W^{1,\infty}}}{\|f_{j-1}\|_{\dot{H}^{s+1}}}\right).
\end{equation}
Next, note that straightforward estimates similar to previous computations yield
\begin{equation}
    \|f_{j-1}\|_{W^{1,\infty}} \le C_2\|f_0\|_{W^{1,\infty}}\prod_{n=1}^{j-1} K_n^2,
\end{equation}
while interpolation and Lemma~\ref{lem:discretegrowth} give 
$$\|f_{j-1}\|_{\dot{H}^{1+s}} \ge \frac{\|f_{j-1}\|_{\dot{H}^1}^{1+s}}{\|f_{j-1}\|_{L^2}^{s}} = \frac{\|f_{j-1}\|_{\dot{H}^1}^{1+s}}{\|f_0\|_{L^2}^{s}} \ge c_1 \|f_0\|_{\dot{H}^1}\prod_{n=1}^{j-1} K_n^{2 + 2s} $$
for a constant $c_1$ that depends only on an upper bound for $\|f_0\|_{\dot{H}^1}/\|f_0\|_{L^2}$. Thus, there is a constant $C_3$ depending only on an upper bound for $\|f_0\|_{\dot{H}^1}/\|f_0\|_{L^2}$ such that
\begin{equation}
\frac{\|f_{j-1}\|_{W^{1,\infty}}}{\|f_{j-1}\|_{\dot{H}^{1+s}}} \le C_3\frac{\|f_0\|_{W^{1,\infty}}}{\|f_0\|_{\dot{H}^1}}\prod_{n=1}^{j-1} K_n^{-2s}.
\end{equation}
Putting this bound into \eqref{eq:fullcomp} we obtain 
\begin{equation}\label{eq:iterate1}
\|f_{j}\|_{\dot{H}^{s+1}} \le K_j^{2s+2}\|f_{j-1}\|_{\dot{H}^{s+1}}\left(1+C_1K_j^{-s} + C_3\frac{\|f_0\|_{W^{1,\infty}}}{\|f_0\|_{\dot{H}^1}}\left(N_j\prod_{n=1}^{j-1} K_n^{-2}\right)^s\right).
\end{equation}
Since $K_j \ge j^{5/2}$ and $s > 2/5$, we have that $\sum_{j=1}^\infty K_j^{-s} < \infty$. Moreover, as $N_j = 2^j$ and $\prod_{n=1}^{j-1} K_n^2 \ge (j-1)!$ we clearly have
$$\sum_{n=1}^\infty \left(N_j \prod_{n=1}^{j-1} K_n^{-2}\right)^s < \infty. $$
Thus, by iterating \eqref{eq:iterate1} we see that there is a constant $C_4$ depending only on $s$ and an upper bound for $\|f_0\|_{\dot{H}^1}/\|f_0\|_{L^2}$ such that for every $j \in \N$ we have
$$\|f_{j}\|_{\dot{H}^{s+1}} \le \exp\left(C_4\left(1+\frac{\|f_0\|_{W^{1,\infty}}}{\|f_0\|_{\dot{H}^1}}\right)\right)\|f_0\|_{\dot{H}^{s+1}}\prod_{n=1}^j K_n^{2s+2}.$$
\end{proof}

\subsection{Concluding the proof of Theorem~\ref{thrm:main}} \label{sec:finish}

With Lemmas~\ref{lem:H1growth} and~\ref{lem:uppermain} at hand, the proof of Theorem~\ref{thrm:main} is essentially immediate from Proposition~\ref{prop:ADcriteria}.

\begin{proof}[Proof of Theorem~\ref{thrm:main}]
It suffices to prove the anomalous dissipation portion of the statement, as once this is established the non-uniqueness for a suitably modified velocity field follows as in \cite{DEIJ22}. Let $u:[0,T_*] \times \T^2 \to \R^2$ be as defined in Section~\ref{sec:udefinition} with the parameters chosen as in Section~\ref{sec:parameters} and the constant $M \ge 2$ picked large enough so that the conclusion of Lemma~\ref{lem:H1growth} holds. Recall that by a simple scaling argument it is sufficient to prove Theorem~\ref{thrm:main} for the particular time $T = T_*$. As described in Section~\ref{sec:parameters}, the fact that $u$ has the regularity claimed in Theorem~\ref{thrm:main} follows easily from the sufficient condition \eqref{eq:velconditions} and the definitions of the parameters $\alpha_j$, $N_j$, and $t_j$. Fix any $s \in (2/5,1/2)$ and mean-zero initial data $f_0 \in H^{1+s} \cap W^{1,\infty}$. We apply Proposition~\ref{prop:ADcriteria} with $\sigma = 1+s$. The fact that the first hypothesis holds is the content of Lemma~\ref{lem:H1growth}, specifically \eqref{eq:blowup}. For the second hypothesis, it follows from Lemmas~\ref{lem:H1growth} and~\ref{lem:uppermain} that for all $t \in [0,T_*)$ the solution of \eqref{eq:AE} satisfies
$$ \|f(t)\|_{L^2}^s \|f(t)\|_{\dot{H}^{1+s}} \le C_1 \|f(t)\|_{\dot{H}^1}^{1+s} \quad \text{with} \quad  C_1 = \frac{C}{c^{1+s}}\frac{\|f_0\|^s_{L^2}\|f_0\|_{\dot{H}^{1+s}}}{\|f_0\|^{1+s}_{\dot{H}^{1}}},$$
where $c$ and $C$ are as in the statements of the lemmas, and in particular only depend on an upper bound for $\|f_0\|_{W^{1,\infty}}/{\|f_0\|_{L^2}}$. Thus, Theorem~\ref{thrm:main} follows from Proposition~\ref{prop:ADcriteria} and there is a lower bound for the amount of energy dissipated on $[0,T_*]$ with the dependencies claimed in Remark~\ref{rem:constants}.
\end{proof}

\subsection{Proof of Theorem~\ref{thrm:2}} \label{sec:thrm2}

We now prove Theorem~\ref{thrm:2}, which amounts to appropriately choosing the parameters $(\alpha_j, N_j, t_j)$ and estimating the Lipschitz norm of the solution to the full advection-diffusion equation. Before beginning the proof, we notice that a careful reading of the proof of Theorem~\ref{thrm:main} above shows that the only properties of the parameters defining the velocity that we needed to deduce anomalous dissipation for every mean-zero $f_0 \in H^{s+1}$ and some fixed $s \in (0,1/2)$ were the following:

\begin{itemize}
    \item $\sum_{j=1}^\infty t_j < \infty$;
    \item $\sum_{j=1}^\infty t_j \prod_{n=1}^{j-1}K_n^4 = \infty$;
    \item $\sum_{j=1}^\infty K_j^{-s} < \infty$;
    \item $\sum_{j=1}^\infty \left(N_j \prod_{n=1}^{j-1} K_n^{-2}\right)^s < \infty$;
    \item there exists $C \ge 1$ such that $1 \le K_{j+1}/K_j \le C$ and for some $M$ sufficiently large $K_j \ge M$;
    \item $N_j$ is an integer.
\end{itemize}

\begin{proof}[Proof of Theorem~\ref{thrm:2}]

Fix $\alpha \in (0,1)$ and $0 \le \beta < (1-\alpha)/2$. For $\epsilon \in (0,1/6)$  to be chosen sufficiently small depending on the gap between $\beta$ and $(1-\alpha)/2$, we apply the proof of Theorem~\ref{thrm:main} with $s = 1/2 - \epsilon$ and the parameters $(\alpha_j, N_j, t_j)$ chosen as 
$$ N_j = j^{4j}, \quad \alpha_j = N_j^{-\alpha}, \quad \text{and} \quad t_j = Mj^{\frac{2}{1-3\epsilon}} j^{-4(1-\alpha)j}, $$
where $M$ is taken sufficiently large. We define our velocity field as in Section~\ref{sec:udefinition} with the parameters as given above. We have $K_j = Mj^{\frac{2}{1-3\epsilon}}$ and it is straightforward to check using Stirling's formula that each of the six conditions stated before the start of the proof are satisfied. The anomalous dissipation claimed in Theorem~\ref{thrm:2} then follows from the proof of Theorem~\ref{thrm:main}.

It remains only to verify the regularity of the solutions $f^\kappa$ claimed in \eqref{eq:scalarregularity}. To estimate $\|f^\kappa(t)\|_{C^\beta}$ we will bound $\|\grad f^\kappa(t)\|_{L^\infty}$ and then interpolate with $\|f^\kappa(t)\|_{L^\infty}$. Let $\{T_j\}_{j \ge 0}$ be as defined in the proof of Theorem~\ref{thrm:main}. For $j \in \N$ and $t \in [T_{j-1}, T_{j-1} + t_j)$ we have
\begin{align}
    & \partial_t (\partial_x f^\kappa) + \alpha_j \psi\left(\frac{t-T_{j-1}}{t_j}\right) S(N_j y)\partial_x (\partial_x f^\kappa) = \kappa \Delta (\partial_x f^\kappa), \label{eq:viscous1} \\ 
    & \partial_t (\partial_y f^\kappa) + \alpha_j \psi\left(\frac{t-T_{j-1}}{t_j}\right) S(N_j y)\partial_x (\partial_y f^\kappa) = \kappa \Delta (\partial_y f^\kappa) - \frac{1}{t_j}\psi\left(\frac{t-T_{j-1}}{t_j}\right) K_j S'(N_j y)\partial_x f^\kappa. \label{eq:viscous2}
\end{align}
From \eqref{eq:viscous1} and the maximum principle it follows that 
\begin{equation} \label{eq:viscous3}
\sup_{t \in [T_{j-1},T_{j-1}+t_j]}\|\partial_x f^\kappa(t)\|_{L^\infty} \le \|\partial_x f^\kappa(T_{j-1})\|_{L^\infty}. 
\end{equation}
Then, treating the term involving $\partial_x f^\kappa$ on the right-hand side of \eqref{eq:viscous2} as a forcing term and using \eqref{eq:viscous3} together with the maximum principle again we obtain 
\begin{equation} \label{eq:viscous4}
    \sup_{t \in [T_{j-1},T_{j-1}+t_j]} \|\partial_y f^\kappa(t)\|_{L^\infty} \le \|\partial_y f^{\kappa}(T_{j-1})\|_{L^\infty} + K_j \|\partial_x f^\kappa(T_{j-1})\|_{L^\infty}.
\end{equation}
Combining the previous two bounds we have $$\sup_{t \in [T_{j-1},T_{j-1} + t_j]}\|\grad f^\kappa(t)\|_{L^\infty} \le (2+K_j)\|\grad f^{\kappa}(T_{j-1})\|_{L^\infty}.$$
Applying the same argument on the time interval $[T_{j-1} + t_j, T_{j})$ gives 
\begin{equation} \label{eq:viscous5}
    \sup_{t \in [T_{j-1},T_{j}]}\|\grad f^\kappa(t)\|_{L^\infty} \le (K_j^2 + C_1K_j)\|\grad f^{\kappa}(T_{j-1})\|_{L^\infty}
\end{equation}
for some $C_1 > 0$. Since $\sum_{j=1}^\infty K_j^{-1} < \infty$, by iterating \eqref{eq:viscous5} and then interpolating the resulting estimate with $\|f^\kappa(t)\|_{L^\infty} \le \|f_0\|_{L^\infty}$ we conclude there is $C_2 > 0$ depending on $f_0$ such that 
\begin{equation} \label{eq:viscous6}
    \sup_{t \in [T_{j-1},T_{j}]}\|f^\kappa(t)\|_{C^\beta(\T^2)} \le C_2 \prod_{n=1}^j K_n^{2\beta} = C_2 M^{2\beta j}(j!)^\frac{4\beta}{1-3\epsilon}.
\end{equation}
Thus, we have 
\begin{equation}
\|f^\kappa\|_{L^2([0,T_*];C^\beta(\T^2))}^2 \le 2 C_2^2\sum_{j=1}^\infty t_j M^{4\beta j} (j!)^\frac{8\beta}{1-3\epsilon} \le 2MC_2^2 \sum_{j=1}^\infty j^4 M^{2\beta j} j^{-4(1-\alpha)j}(j!)^\frac{8\beta}{1-3\epsilon} < \infty
\end{equation}
provided that $\epsilon > 0$ is chosen small enough so that 
$$ \frac{1}{1-3\epsilon} < \frac{1-\alpha}{2\beta}.$$
    
\end{proof}

\section{General Remarks and Questions}

Let us close by giving a few questions for further investigation that might be interesting. 

\subsection{Uniqueness Threshold?}

Using the construction here, it appears that the strongest that the modulus of continuity of the velocity can be is $x|\log(x)|^2$. It is not clear whether one can reach the Osgood threshold using this technique or whether it is even possible. While it is clear that anomalous dissipation is impossible when the velocity field satisfies the Osgood condition, it is not clear that the Osgood condition is really necessary even for uniqueness in the PDE setting. Resolving this gap, for non-uniqueness and/or anomalous dissipation, may be of great theoretical interest.  

\subsection{Autonomous flows}
The example given here relies on the existence of a ``singular time" in the velocity field. While this may be interesting for the purposes of studying anomalous dissipation coming from a potential finite-time singularity in the fluid equations, it is of great mathematical, and possibly physical, interest to construct autonomous flows that give anomalous dissipation. There are two relevant directions that one can think about here. In two dimensions, it is possible that the examples given by Alberti, Bianchini, and Crippa \cite{ABC} or similarly constructed flows can give anomalous dissipation for some or all data. In three dimensions and higher, the existence of an autonomous flow giving anomalous dissipation (and non-uniqueness) for a single data or a large class of data is immediate from the construction here and earlier works, for instance by treating the third dimension as a time variable. There is, however, no example of an autonomous flow in three dimensions giving anomalous dissipation for \emph{all} smooth data. It is possible that one can lift versions of the flow constructed here to serve this purpose. 

\subsection{Forwards-Backwards Principle}

Lemma \ref{fwdbckwdp} gives a sufficient condition for exponential growth of solutions to the transport equation with a time-periodic velocity (the mapping $\Phi$ can be taken to be the associated Lagrangian flow at $t=T_*$, the period of the velocity field). It is not clear whether there exists any time-periodic and \emph{smooth} velocity field for which such an inequality holds.

\section*{Acknowledgements}

T.M.E. acknowledges funding from NSF DMS-2043024 and the Alfred P. Sloan Foundation. K.L. acknowledges funding from NSF DMS-2038056. The authors thank T. Drivas for helpful conversations and suggestions. 

\appendix
\section{Interpolation and commutator estimates} \label{appendix}

In this section we recall some Sobolev interpolation and commutator estimates that are needed in the proof of Lemma~\ref{lem:uppermain}. 

The result from interpolation theory that we require is as follows. For a proof, see e.g. \cite[Corollary 3.2]{chandler2015interpolation}). 

\begin{lemma}[Sobolev space interpolation] \label{lem:interp}
    Fix $0 \le s_0 < s_1 < \infty$ and let $T:\dot{H}^\sigma(\T^2) \to \dot{H}^{\sigma}(\T^2)$ be a bounded, linear operator for $\sigma \in \{s_0, s_1\}$. Then, for every $\theta \in (0,1)$, defining $s_\theta = \theta s_0 + (1-\theta)s_1$ we have that $T:\dot{H}^{s_\theta}(\T^2)\to \dot{H}^{s_\theta}(\T^2)$ is also bounded and 
    $$ \|T\|_{\dot{H}^{s_\theta} \to \dot{H}^{s_\theta}} \le \|T\|_{\dot{H}^{s_0}\to \dot{H}^{s_0}}^\theta \|T\|_{\dot{H}^{s_1}\to \dot{H}^{s_1}}^{1-\theta}. $$
\end{lemma}

Next, we have a homogeneous Kato-Ponce inequality; see for instance \cite[Problem 2.7]{MuscaluSchlag}.

\begin{lemma}\label{lem:Kato}
    Fix $s \in (0,1)$. There exists a constant $C_s$ such that for every $f,g \in C^\infty(\T^2)$ there holds 
    $$\|D^s(f g) - f D^s g\|_{L^2} \le C_s \|D^s f\|_{L^2}\|g\|_{L^\infty}.$$
\end{lemma}

\addcontentsline{toc}{section}{References}
	\bibliographystyle{abbrv}
	\bibliography{ADbib}

\end{document}